\documentclass[11pt]{amsart}

\usepackage{amsmath}
\usepackage{amssymb}
\usepackage{graphicx}

\hoffset = - 0.5in

\newtheorem{theorem}{Theorem}[section]

\newtheorem{proposition}[theorem]{Proposition}
\theoremstyle{definition}
\newtheorem{definition}[theorem]{Definition}

\numberwithin{equation}{section}

\renewcommand{\geq}{\geqslant}
\renewcommand{\leq}{\leqslant}

\title{$D(X) < 1$ or $\hat{D}(X) < 1$ imply Property($K$)}

\author[Tim Dalby]{Tim Dalby}

\date{\today}



\email{tim\_dalby@bigpond.com}

\begin{document}

\parindent = 0pt
\parskip = 8pt

\begin{abstract}

Two new Banach space moduli, that involve weak convergent sequences, are introduced.  It is shown that if either one of these moduli are strictly less than 1 then the Banach space has 
Property($K$).

\end{abstract}

\maketitle

\section{Introduction}

A Banach space, $X$, has the weak fixed point property, w-FPP, if every nonexpansive mapping, $T$, on every weak compact convex nonempty subset, $C$, has a fixed point.  The past forty or so years has seen a number of Banach space properties shown to imply the w-FPP.  Some such properties are weak normal structure, Opial's condition, Property($K$)  and Property($M$).  Here two new moduli are introduced and are linked to one of these properties, Property($K$).  More information on the w-FPP and associated Banach space properties and moduli can be found in [3].

The key definitions and terminology are below.

\begin{definition}

Sims, [6]

A Banach space $X$ has property($K$) if there exists $K \in [0, 1)$ such that whenever $x_n \rightharpoonup 0, \lim_{ n \rightarrow \infty} \| x_n \| = 1 \mbox{ and }  \liminf_{n \rightarrow \infty} \| x_n - x \|  \leq 1 \mbox{ then } \| x \| \leq K.$

\end{definition}

\begin {definition}

Opial [5] 

A Banach space has Opial's condition if
\[ x_n \rightharpoonup 0 \ \mbox {and } x \not = 0 \mbox { implies } \limsup_n \| x_n \| < \limsup_n \| x_n - x \|. \]

The condition remains the same if both the $\limsup$s are replaced by $\liminf$s.

\end {definition}

Later a modulus was introduced to gauge the strength of Opial's condition and a stronger version of the condition was defined.

\begin{definition}

Lin, Tan and Xu, [4]

Opial's modulus is
\[ r_X(c) = \inf \{ \liminf_{n\rightarrow \infty} \| x_n - x \| - 1: c \geq 0, \| x \| \geq c,  x_n \rightharpoonup 0 \mbox{ and }\liminf_{n\rightarrow \infty} \| x_n \| \geq 1 \}. \]

$X$ is said to have uniform Opial's condition if $r_X(c) > 0$ for all $c > 0.$  See [4] for more details.

\end{definition}

There is a direct link between Opial's modulus and Property($K$).  Dalby proved in [1] that $r_{X}(1) > 0$ is equivalent to $X$ having Property($K$).  This will be used in the next section.

The two new moduli are defined next.

\begin{definition}

Let $X$ be a Banach space.  Let

\[ D(X) = \sup\{ \liminf_{n \rightarrow \infty}\| x_n - x \|: x_n \rightharpoonup x, \| x_n \|  = 1 \mbox{ for all } n\} \]
and let
\[ \hat{D}(X) = \sup\{ \| x \|: x_n \rightharpoonup x, \| x_n \| = 1 \mbox{ for all } n\}. \]

\end{definition}

So $0 \leq D(X) \leq 2$ and $ \hat{D}(X) \leq 1.$

Some values for $D(X)$ are $D(\ell_1) = 0, D(c_0) = 1 \mbox{ and }  D(\ell_p) = 2^{1/p}.$

The reason that these two moduli are introduced is that in [2] Dalby showed that if in the dual, $X^*,$ a certain weak* convergent sequence, $(w_n^*),$ satisfies either one of two properties then $X$ satisfied the w-FPP.  Let $w_n^* \stackrel{*}{\rightharpoonup} w^* \mbox { where } \| w^* \| \leq 1$ then if $w^*$  is `deep' within the dual unit ball or $w_n^* - w^*$ eventually `deep' within the dual unit ball then $X$ has the w-FPP. So $D(X^*) < 1$ or $\hat{D}(X^*) < 1$ ensures this.

The w-FPP is known to be separably determined so all Banach spaces are assumed to be separable.

\section{Results}

\begin{proposition}

Let $X$ be a separable Banach space.  If $D(X) < 1$ then $r_X(1) > 0.$  That is, $X$ has Property($K$).

\end{proposition}

\begin{proof}

Let $x_n \rightharpoonup 0, \liminf_{n \rightarrow \infty}\| x_n \| \geq 1 \mbox{ and } \| x \| \geq 1.$

Using the lower semi-continuity of the norm, $\liminf_{n \rightarrow \infty}\| x_n + x \| \geq \| x \| \geq 1.$  By taking subsequences if necessary we may assume that $\| x_n + x \| \not = 0$ for all $n.$

Now $ \left \| \frac{\displaystyle {x_n + x}}{ \displaystyle {\| x_n + x \|}} \right \| = 1 \mbox{ for all } n, {\frac{\displaystyle{x_n + x}}{ \displaystyle{\| x_n + x \| }}} \rightharpoonup \frac{\displaystyle{x}}{ \displaystyle \liminf_{n \rightarrow \infty}\| x_n + x \| }.$

For ease of reading let $\alpha = \liminf_{n \rightarrow \infty}\| x_n + x \|.$

Then 

\begin{align*}
  \liminf_{n \rightarrow \infty}\left \| \frac{\displaystyle x_n + x}{\| x_n + x \|} - \frac{\displaystyle x}{\displaystyle \alpha} \right \| 
 & = \liminf_{n \rightarrow \infty} \frac{\displaystyle 1}{\| x_n + x \|} \liminf_{n \rightarrow \infty}\left \| x_n + x - \| x_n + x \| \frac{\displaystyle x}{\displaystyle \alpha} \right \|  \\
 & =  \frac{\displaystyle 1}{\alpha} \liminf_{n \rightarrow \infty}\left \| x_n + x - \| x_n + x \| \frac{\displaystyle x}{\displaystyle \alpha} \right \|  \\
 & =  \frac{\displaystyle 1}{\alpha}  \liminf_{n \rightarrow \infty}\left \| x_n - \left (\frac{\displaystyle \| x_n + x \|}{ \alpha} - 1 \right ) x \right \| \\
 & \geq \frac{\displaystyle 1}{\alpha} \left | \liminf_{n \rightarrow \infty} \| x_n \|  - \liminf_{n \rightarrow \infty} \left | \frac{\displaystyle \| x_n + x \|}{ \alpha} - 1 \right | \| x \| \right | \\
 & =  \frac{\displaystyle 1}{\alpha} \left | \liminf_{n \rightarrow \infty} \| x_n \|  + \left | \frac{\displaystyle \alpha}{\alpha} - 1 \right | \| x \| \right ) \\
 & = \frac{\displaystyle 1}{\alpha}  \liminf_{n \rightarrow \infty} \| x_n \|  \\
 & \geq \frac{\displaystyle 1}{\alpha} \\
 & = \frac{\displaystyle 1}{\liminf_{n \rightarrow \infty}\| x_n + x \|}.
 \end{align*}

We have 

\[ D(X) \geq \liminf_{n \rightarrow \infty}\left \| \frac{\displaystyle x_n + x}{\displaystyle \| x_n + x \|} - \frac{\displaystyle x}{\displaystyle \liminf_{n \rightarrow \infty}\| x_n + x \|} \right \| \geq \frac{1}{\displaystyle \liminf_{n \rightarrow \infty} \| x_n + x \|}.  \qquad \dag \]

So $\liminf_{n \rightarrow \infty} \| x_n + x \| \geq \frac{\displaystyle 1}{\displaystyle D(X)}.$

This means that $r_X(1)  + 1 \geq \frac{\displaystyle 1}{\displaystyle D(X)} \mbox{ or } r_X(1) \geq \frac{\displaystyle 1}{\displaystyle D(X)} - 1 > 0.$

\end{proof}

A second way to prove this proposition is via a contradiction as shown below.

\begin{proof}

Assume that $D(X) < 1 \mbox{ and } r_X(1) \not > 0.$  Then $r_X(1) = 0.$  

Given $\epsilon > 0 $ there exists a sequence $(x_n) \mbox{ in } X \mbox{ where } x_n \rightharpoonup 0, \newline \liminf_{n \rightarrow \infty}\| x_n \| \geq 1 \mbox{ and } x \in X, \| x \| \geq 1 \mbox{ such that } \liminf_{n \rightarrow \infty}\| x_n + x \| < 1 + \epsilon.$

Therefore $1 \leq \| x \| \leq \liminf_{n \rightarrow \infty}\| x_n + x \| < 1 + \epsilon.$  So apart from the last inequality the set up is the same as in the previous proof and this proof follows the same pathway.  So now jumping to a line above, the one labeled with \dag, 

\[ D(X)  \geq \frac{1}{\displaystyle \liminf_{n \rightarrow \infty} \| x_n + x \|} > \frac{1}{ \displaystyle 1 + \epsilon }. \]

Letting $\epsilon \rightarrow 0 \mbox{ gives } D(X) \geq 1 \mbox{ but } D(X) < 1.$

So the desired contradiction is arrived at.
 
\end{proof}

Next is the second moduli's turn.

\begin{proposition}
Let $X$ be a separable Banach space.  If $\hat{D}(X) < 1$ then $r_X(1) > 0.$  That is, $X$ has Property($K$).
\end{proposition}

\begin{proof}

Let $x_n \rightharpoonup 0, \liminf_{n \rightarrow \infty}\| x_n \| \geq 1 \mbox{ and } \| x \| \geq 1.$

Now $x_n + x \rightharpoonup x \mbox{ so } \liminf_{n \rightarrow \infty}\| x_n + x \| \geq \| x \| \geq 1.$  Without loss of generality we may assume $\| x_n + x \| \not = 0$ for all $n.$

Then $\left \| \frac{\displaystyle x_n + x}{\displaystyle \|x_n + x \|} \right \| = 1 \mbox{ for all } n, \frac{\displaystyle x_n + x}{\displaystyle \|x_n + x \|} \rightharpoonup \frac{\displaystyle x}{ \displaystyle \liminf_{n \rightarrow \infty}\| x_n + x \| }.$

Hence $1 > \hat{D}(X) \geq \frac{\displaystyle \| x \|}{\displaystyle \liminf_{n \rightarrow \infty}\| x_n + x \| }$ leading to

\begin{align*}
\| x \| & \leq \liminf_{n \rightarrow \infty}\| x_n + x \| \hat{D}(X) \\
\liminf_{n \rightarrow \infty}\| x_n + x \| & \geq \frac{\displaystyle \| x \|}{\displaystyle \hat{D}(X)} \\
& \geq \frac{\displaystyle 1}{\displaystyle \hat{D}(X)} \\
\mbox{ Thus } r_X(1) + 1 & \geq \frac{\displaystyle 1}{\displaystyle \hat{D}(X)} \\
r_X(1) & > \frac{\displaystyle 1}{\displaystyle \hat{D}(X)} - 1 \\
& > 0.
\end{align*}

\end{proof}

A second way to prove this proposition is by finding a value of $K$ for Property($K$).

\begin{proof}

Let $x_n \rightharpoonup 0, \| x_n \| = 1 \mbox{ for all } n \mbox{ and }\liminf_{n \rightarrow \infty}\| x_n - x \| \leq 1.$

If $\liminf_{n \rightarrow \infty}\| x_n - x \| = 0$ then because $\| x \| \leq \liminf_{n \rightarrow \infty}\| x_n - x \|$ we have $x = 0$ and $K$ can be taken as zero.

So assume $\liminf_{n \rightarrow \infty}\| x_n - x \| > 0$ and by taking subsequences if necessary, assume $\| x_n - x \| \not = 0$ for all $n.$

Using the same argument as in the previous proof 

\[\| x \| \leq \liminf_{n \rightarrow \infty}\| x_n - x \| \hat{D}(X) \leq \hat{D}(X) < 1.\]

So $K$ can be taken as $\hat{D}(X).$

\end{proof}

\end{document}